\documentclass[11pt]{amsart}
\usepackage{amsmath,amsfonts,amsthm,amssymb,amscd, verbatim, graphicx}
\usepackage{hyperref}
\usepackage[notref, notcite]{showkeys}

\newcommand{\trace}{\mathrm{Tr}}

\newcommand{\comments}[1]{}
\newcommand{\stab}{\mathrm{Stab}}
\renewcommand{\leq}{\leqslant}
\renewcommand{\geq}{\geqslant}

\newcommand{\R}{\mathbb{R}}
\newcommand{\graphg}{\mathcal{G}}

\newcommand{\qr}{\mathrm{QR}}
\newcommand{\Sch}{\mathrm{Sch}}

\newtheorem{prop}{Proposition}[section]
\newtheorem{thm}{Theorem}
\newtheorem{conj}[prop]{Conjecture}
\newtheorem{cor}[prop]{Corollary}
\newtheorem{lem}[prop]{Lemma}

\theoremstyle{definition}

\numberwithin{equation}{section}

\begin{document}

\title{Quasirandom group actions}

\author{N. Gill}
\address{Department of Mathematics, The Open University, Walton Hall, Milton Keynes, MK7 6AA, UK}
\email{n.gill@open.ac.uk}

\begin{abstract}
Let $G$ be a finite group acting transitively on a set $\Omega$. We study what it means for this action to be {\it quasirandom}, thereby generalizing Gowers' study of quasirandomness in groups. We connect this notion of quasirandomness to an upper bound for the convolution of functions associated with the action of $G$ on $\Omega$. This convolution bound allows us to give sufficient conditions such that sets $S,T\subset G$ and $\Gamma\subseteq \Omega$ contain elements $s\in S, t\in T, \gamma\in\Gamma$ such that $s(\gamma)=t$. Other consequences include an analogue of `the Gowers trick' of Nikolov and Pyber for general group actions, a sum-product type theorem for large subsets of a finite field, as well as applications to expanders and to the study of the diameter and width of a finite simple group.
\end{abstract}

\maketitle

In his seminal 2008 paper entitled ``Quasirandom groups'', Gowers introduced the notion of a {\it $d$-quasirandom group}. He gives a number of formulations of this idea but, for our purposes, it is easiest to define a group $G$ to be {\it $d$-quasirandom} (for some $d\in \mathbb{R}^+$) if every non-trivial irreducible representation of $G$ has dimension at least $d$. Gowers related this definition of quasirandomness to notions of quasirandomness for functions $G\to \mathbb{R}$, and for particular graphs related to $G$ (`directed Cayley graphs'). These connections allowed him to prove the following fundamental result:

\begin{thm}\label{t: quasirandom group}
 Let $G$ be a finite $d$-quasirandom group of order $n$. Let $A$, $B$ and $C$ be three subsets of $\Gamma$ such that $|A|\cdot |B|\cdot |C| > n^3 /d$.
Then there exist $a \in A, b \in B$ and $c \in C$ with $ab = c$. 
\end{thm}

Some time after Gowers proved this result, Babai, Nikolov and Pyber were able to give a different proof. They proved a bound for the convolution of probability measures on $G$ and showed that Theorem~\ref{t: quasirandom group} followed directly. What is more the convolution bound had a number of other important applications, most notably to the theory of expander graphs.

In this paper we generalize Theorem~\ref{t: quasirandom group}. We show that it is a particular case of a result concerning arbitrary transitive actions $G$ on a set $\Omega$. Our method involves a careful study of the original arguments of Gowers, and of Babai-Nikolov-Pyber. We are able to adapt both arguments to give different bounds on the convolution of functions related to the action, and these bounds imply the mentioned generalization of Theorem~\ref{t: quasirandom group}, as well as a number of other significant results.

\section{Main results}

In order to state our main results we must establish some notation which will hold throughout the paper. First we set $G$ to be a finite group acting transitively on a finite set $\Omega$. 

Consider two functions $X:G\to\mathbb{R}$ and $Y:\Omega\to\mathbb{R}$. We define the \emph{convolution} $X\ast_c Y$ of $X$ and $Y$ to be the following function on $\Omega$:
\begin{equation}\label{e: convolution definition}
(X\ast_c Y)(\omega)=\sum\limits_{g\in G} X(g)Y(g^{-1}\omega).
\end{equation}
This definition, which has appeared in various places in the literature, is a generalization of the definition of convolution given in \cite{bnp}. Observe that if $X$ and $Y$ are probability distributions then $X\ast_cY$ is also a probability distribution; on the other hand if either $X$ or $Y$ sum to $0$ then $X\ast_c Y$ sums to $0$.

We write $H=\stab_G(\omega)$, the stabilizer in $G$ of some element $\omega\in \Omega$. If $\chi$ is a representation of $H$, then we write $\chi_H^G$ for the representation of $G$ induced from $\chi$. The representation $1_H^G$ is the {\it permutation representation} of $G$ on the (left) cosets of $H$. We set $d_H$ to be the minimum degree of a non-trivial irreducible component of the representation $1_H^G$; similarly $m_H$ is the minimum multiplicity of a non-trivial irreducible component of the representation $1_H^G$.

We are now able to state two theorems about convolutions, both of which are proved in \S\ref{s: bnp}. The first is a generalization of \cite[Theorem 2.1]{bnp} and is couched in terms of probability distributions.\footnote{L. Pyber has pointed out to me that a result similar to Theorem~\ref{t: main convolution} has been proved by B. Szegedy \cite[Corollary 1.3]{szegedy}. Szegedy's result applies not just to finite groups, but more generally to compact Hausdorff topological groups.}

\begin{thm}\label{t: main convolution}
Let $G$ be a finite group acting transitively on a set $\Omega$ and let $X$ be a probability distribution over $G$, $Y$ a probability distribution over $\Omega$. Then
\begin{equation}\label{e: main convolution}
\Vert X\ast_c Y - U_\Omega \Vert \leq  \sqrt{|G|/d_H}\cdot \Vert X-U_G\Vert\cdot \Vert Y-U_\Omega\Vert,
\end{equation}
where $U_G$ (resp. $U_\Omega$) is the uniform probability distribution on $G$ (resp. on $\Omega$).
\end{thm}

The second convolution theorem is a generalization of \cite[Lemma 3.2]{gowers}. \footnote{Gowers did not state his original result in terms of convolution of functions, but he could have if he'd wanted to.} 

\begin{thm}\label{t: main convolution 2}
Let $G$ be a finite group acting transitively on a set $\Omega$, let $S$ be any subset of $G$, let $\chi_S:G\to\mathbb{R}$ be the characteristic function of $S$ and let $f:\Omega\to \mathbb{R}$ be a function that satisfies $\sum_{x\in G}f(x)=0$. Then
\begin{equation}\label{e: main convolution 2}
\|\chi_S\ast_c f\|\leq \sqrt{\ell_S |\Omega|/m_H}\cdot \|\chi_S\|\cdot \|f\|.
\end{equation}
where $\ell_S=\max\{ |g_1Hg_2\cap S| \, \mid \, g_1,g_2\in G\}.$
\end{thm}

The results of Gowers and Babai-Nikolov-Pyber that these two theorems generalize both pertain to the (left) regular action of $G$ on itself. For this action the distinction between $d_H$ and $m_H$ is lost, as both are equal to the minimum dimension of a non-trivial irreducible representation of $G$. 
The two theorems, then, highlight one of the main differences between the approach of Gowers (where bounds involve multiplicity, and dimension enters only by virtue of its connection to dimension) and the approach of Babai-Nikolov-Pyber (where bounds involve dimension directly).

\subsection{General consequences}

Theorems~\ref{t: main convolution} and \ref{t: main convolution 2} have a number of general consequences for subsets connected to group actions. The first of these is an analogue of the main result of \cite{bnp} which is itself a variant on the original ``Gowers Trick''.

\begin{thm}\label{t: main result}
Let $G$ be a finite group acting transitively on a set $\Omega$, let $S \subseteq G$ and let $\Gamma\subseteq \Omega$. Then the following two inequalities hold:
\begin{equation}\label{e: main sets}
|S(\Gamma)| > \frac{|\Omega|}{1+\frac{|G||\Omega|}{d_H|S||\Gamma|}} \geq \min\left\{ \frac{|\Omega|}2, \frac {d_H|S||\Gamma|}{2|G|}\right\};
\end{equation}
\begin{equation}\label{e: main sets 2}
|S(\Gamma)| > \frac{|\Omega|}{1+\frac{\ell_S|\Omega|^2}{m_H|S||\Gamma|}} \geq \min\left\{ \frac{|\Omega|}2, \frac {m_H|S||\Gamma|}{2\ell_S|\Omega|}\right\}.
\end{equation}
In particular if $k$ is a positive number and $|S|\geq \min\{k\frac{|G|}{d_H}, k\frac{\ell_S|\Omega|}{m_H}\}$, then $|S(\Gamma)|> \frac12 \min\{|\Omega|, k|\Gamma|\}$.
\end{thm}

Note that, here and elsewhere, we write group actions on the left. In particular $S(\Gamma)=\{s\gamma \mid s\in S, \gamma\in\Gamma\}.$ Recall that $\ell_S$ was defined in the statement of Theorem~\ref{t: main convolution 2}.

Theorem~\ref{t: main result} has a number of consequences. The first is the generalization of Theorem~\ref{t: quasirandom group} that we mentioned at the start of this paper.

\begin{cor}\label{c: quasirandom main}
Let $G$ be a finite group acting transitively on a set $\Omega$. Suppose that any non-trivial irreducible component of the corresponding permutation representation has degree at least $d_H$. Let $S$ be a subset of $G$ and $\Delta_1, \Delta_2$ subsets of $\Omega$ such that $|S||\Delta_1||\Delta_2|\geq |\Omega|^2|G|/d_H$. Then there exist $g\in S$, $\omega_1\in \Delta_1$ and $\omega_2\in \Delta_2$ such that $g(\omega_1)=\omega_2$.
\end{cor}
\begin{proof}
Write $n$ for $|\Omega|$. The inequality $|S||\Delta_1||\Delta_2|\geq n^2|G|/d_H$, combined with the inequality \eqref{e: main sets}  - setting $\Gamma=\Delta_1$ - implies that
$$|S(\Delta_1)| > \frac{n^2}{n+|\Delta_2|} > n-|\Delta_2|.$$
Now the pigeonhole principle implies that $S(\Delta_1)\cap \Delta_2\neq \emptyset$ and the result follows.
\end{proof}

The results stated so far take on a particularly interesting aspect when the group $H$ is the centralizer of an element $g\in G$. In this case the action of $G$ on $\Omega$ is isomorphic to the action of $G$ on the conjugacy class $C$ which contains $g$. In this context we have the following corollary, the proof of which is given in \S\ref{s: corollaries}.

\begin{cor}\label{c: trick 3}
Let $G$ be a finite group, let $C$ be a conjugacy class of $G$ and let $H$ be the centralizer of an element of $C$. Suppose that $A$ is a subset of $C$ such that
\begin{enumerate}
\item $|A|\geq \frac{|C|}{2}$ and
\item $d_H > \frac{8}{k}|H|\ell_C$ for some positive integer $k$.
\end{enumerate}
Then $(A\cup A^{-1})^{5+10k}\supseteq C$.
\end{cor}

Note that, since $C$ is invariant under conjugation,
$$\ell_C=\max\{|C\cap g_1Hg_2| \, \mid \, g_1, g_2\in G\}=\max\{|C\cap gH| \, \mid \, g\in G\}$$
and note that for the rest of this paper we tend to use the symbol $A$ (rather than $S$) for subsets of $G$ that lie wholly inside a conjugacy class $C$.

Observe that Corollary~\ref{c: trick 3} applies only to very large sets in $C$  - sets that are at least half the size of $C$. In contrast Theorem~\ref{t: main result} can be applied to much smaller sets. In general our method will be to apply Theorem~\ref{t: main result} first, to obtain expansion results for sets up to half the size of $C$, and then to use Corollary~\ref{c: trick 3} to obtain all of $C$.

Effectively, then, we use Corollary~\ref{c: trick 3} much as the original ``Gowers Trick'' of Nikolov and Pyber \cite{npy} is used; moreover our proof of the result is a direct adaptation of that found in \cite{npy}. We have not attempted to optimise the value $5+10k$; a more involved analysis would substantially decrease this value.

\subsection{Consequences for expanders}\label{s: expanders}

Let $X=(V,E)$ be a (directed) graph and $\epsilon>0$ a real number. For a set of vertices $W\subseteq V$, define $\partial W$ to be the number of edges of form $(w,y)$ where $w\in W$ and $y\in V\backslash W$. Now recall that $X$ is called an {\it $\epsilon$-expander}  if
$$\min\left\{\frac{|\partial W|}{|W|}\geq \epsilon \, \mid \, W \subset V, |W|\leq \frac12|V|\right\}.$$

Consider a group $G$ acting transitively on a set $\Omega$ and let $S$ be a subset of $G$. Define the {Schreier graph} $\Sch(G,\Omega, S)$ to be the graph whose vertices are elements of $\Omega$ and whose edges are $(\omega,s\omega)$ for every $\omega\in\Omega$ and every $s\in S$.

We aim to construct infinite families of Schreier graphs, $(X_n) = \Sch(G_n,\Omega_n, S_n)$ (where $n$ varies over $\mathbb{N}$) such that each graph in the family is an $\epsilon$-expander, for some absolute constant $\epsilon$. In this case we say that $(X_n)_{n\in \mathbb{N}}$ is an {\it $\epsilon$-expander family}. We restrict, first of all, to the case where our family consists of graphs which have constant degree $d$ as this is the most interesting (and most difficult).

There are several methods for proving that a given family of Schreier graphs is an $\epsilon$-expander family. The one that interests us here makes use of the product theorems of Helfgott \cite{helfgott2, helfgott3} and its generalizations \cite{bgt2, ps2}. It was developed, first of all, by Bourgain and Gamburd \cite{bg, bgsu2} using ({\it inter alia}) ideas of Sarnak and Xue \cite{sx}.

Yehudoff \cite{yehudayoff} gives a beautiful explanation of how the Bourgain-Gamburd method works: he breaks this method down into three stages, and it is the last of these, `the end game' that is of concern to us here. In order to show that $(X_n)=\Sch(G_n,\Omega_n, S_n)$ is a family of $\epsilon$-expanders for $n\in \mathbb{N}$, one needs to prove a lemma of the following form \cite[Lemma 4]{yehudayoff}:

\begin{lem}\label{l: required for expansion}
There exists a universal constant $c>0$ so that for every $n\in \mathbb{N}$, for every probability distribution $\mu_n$ on $G_n$ and for every function $f_n:\Omega_n\to \mathbb{R}$ that satisfies $\sum_{x\in G_n}f_n(x)=0$,
\begin{equation}\label{e: end game}
\|\mu_n\ast_c f_n\|^2\leq |G_n|^{1-c}\cdot \|\mu_n\|\cdot \|f_n\|.
\end{equation}
\end{lem}

To prove a result of this kind we use Lemma~\ref{l: 32} to adjust Theorem~\ref{t: main convolution} so that it is stated in terms of `functions that sum to $0$'.

\begin{prop}\label{p: expanders}
Let $\mu$ be a probability distribution on $G$ and let $f:\Omega\to \mathbb{R}$ be a function that satisfies $\sum_{x\in G}f(x)=0$. Then
$$\|\mu\ast_c f\|^2\leq |G|/d_H\cdot \|\mu\|\cdot \|f\|.$$
\end{prop}

The proposition has the following immediate corollary.

\begin{cor}\label{c: end game}
Suppose that there exists a constant $c>0$ and a family $(X_n)_{n\in \mathbb{N}}=\Sch(G_n,\Omega_n, S_n)$ of Schreier graphs such that the minimal dimension of an irreducible component of the permutation representation for the action of $G_n$ on $\Omega_n$ is at least $|G_n|^c$. Then \eqref{e: end game} holds.
\end{cor}

This corollary applies to many of the known constructions of $\epsilon$-expander families:
\begin{itemize}
 \item {\bf The (left) regular action of $G$ on itself}: Here $\Omega_n=G_n$ and the Schreier graph is actually a Cayley graph. This is the original setting of Bourgain and Gamburd. Note that once one knows that a Cayley graph is an $\epsilon$-expander, then one can use standard results on eigenvalues of adjacency matrices (including, for instance, \cite[Proposition 11.17]{hlw}) to prove expansion on other Schreier graphs. 
 \item{\bf The action is 2-transitive}: In this case $1_G^H=1+\chi$ where $\chi$ is an irreducible representation, and thus $d_H = |\Omega|-1$. This situation has been studied by Bourgain and Yehudayoff \cite{by} and used to construct a {\it monotone} expander family. Yehudayoff refers to this work in the survey mentioned above, where he also states a special (and weaker) case of Corollary~\ref{c: end game} \cite[Lemma 14]{yehudayoff}.
 \item{\bf Margulis' original family of expanders}: These are expanders corresponding to a family of Schreier graphs $(X_p)_{p \textrm{ a prime}}=\Sch({\rm AGL}_2(p), (\mathbb{Z}/p\mathbb{Z})^2, S_p)$ where $S_p$ is a particular subset of size $8$ in ${\rm AGL}_2(p)$. Again, since ${\rm AGL}_2(p)$ acts $2$-transitively on $(\mathbb{Z}/p\mathbb{Z})^2$, Corollary~\ref{c: end game} applies.
\end{itemize}

Thus, of the known $\epsilon$-expander families, the only ones where Corollary~\ref{c: end game} does not (obviously) apply are those constructed using the zig-zag product pioneered by Reingold, Vadhan and Wigderson \cite{rvw}.

If one relaxes the condition that the family of graphs be $d$-regular, then the following result can be used (along with lower bounds for $d_H$ given by \cite{landseitz}) to obtain infinite families of $\epsilon$-expander families for (say) any given family of simple groups of Lie type.

\begin{cor}\label{c: expanders}
Let $G$ be a finite group acting transitively on a set $\Omega$ and let $H$ be the stabilizer of an element of $\Omega$. Let $\delta>0$ and let $S$ be a subset of $G$ satisfying
$$|S|\geq \min\left\{ \frac{(2+\delta) |G|}{d_H}, \frac{(2+\delta) \ell_S|\Omega|}{m_H} \right\}.$$
Then $\Sch(G,S,\Omega)$ is an $\epsilon$-expander where $\epsilon = \frac{\delta}{4+\delta}$.
\end{cor}
\begin{proof}
Let $\Gamma$ be a subset of $\Omega$ of size at most $\frac12|\Omega|$. The lower bounds on the order of $S$ imply, by Theorem~\ref{t: main result}, that
$$|S(\Gamma)| > \frac{|\Omega|}{1+\frac{|\Omega|}{(2+\delta)|\Gamma|}} = \frac{(2+\delta)|\Omega||\Gamma|}{(2+\delta)|\Gamma|+|\Omega|} \geq \frac{(2+\delta)|\Gamma|}{\frac12(4+\delta)} = \left(1+\frac{\delta}{4+\delta}\right)|\Gamma|.$$
Now $|\partial\Gamma| \geq |S(\Gamma)| - |\Gamma| >\frac{\delta}{4+\delta}|\Gamma|$ and the result follows.
\end{proof}

\subsection{Sum-product}

We remarked in the previous section that our results are particularly effective when we consider a 2-transitive action of a finite group $G$. We study a particular instance of such an action in order to prove the following {\it sum-product result for large sets in finite fields.}

\begin{prop}\label{p: sumsets}
Let $A$ be a subset of $\mathbb{F}_q\backslash\{0\}$.
\begin{enumerate}
 \item If $|A|\geq q^{2/3}$ then $|A+AA|>\frac{q}{2}$.
 \item If $|A|=q^{1/2+\delta}$ for some $\delta\in(0,\frac16)$, then $|A+AA|> \frac12 q^{1/2+3\delta}$.
\end{enumerate}
\end{prop}

\begin{proof}
 We apply Theorem~\ref{t: main result} to the following situation: $G=(\mathbb{F}_q,+)\rtimes (\mathbb{F}_q^*,\cdot)$ acting as a 1-dimensional affine group on $\Omega = (\mathbb{F}_q,+)$. The group $G$ here is isomorphic to $E_q\rtimes C_{q-1}$, a semi-direct product of an elementary-abelian group of order $q$ with a cyclic group of order $q-1$. Observe that, for $(a,b)\in G, c\in \Omega$, 
 \begin{equation}\label{e: sp}
  (a,b)(c)=a+bc.
 \end{equation}
The action of $G$ on $\Omega$ is 2-transitive hence, as we observed in the previous section, $d_H=|\Omega|-1=q-1$. 

Next define sets
$$S= \{(a_1, a_2) \mid a_1,a_2\in A\} \textrm{  and  } \Gamma=A,$$
and observe that \eqref{e: sp} implies that $S(\Gamma)=A+AA$. Now Theorem~\ref{t: main result} can be applied and \eqref{e: main sets} yields that
\begin{equation}\label{e: sp2}
|A+AA|=|S(\Gamma)| > \frac{q}{1+\frac{q(q-1)q}{(q-1)|A|^3}}=\frac{q|A|^3}{|A|^3+q^2}.
\end{equation}
Suppose first that $|A|\geq q^{2/3}$. Then \eqref{e: sp2} implies that $|A+AA|>\frac{q}2$ as required. On the other hand if $|A|=q^{1/2+\delta}$ for some $\delta\in(0,\frac16)$, then \eqref{e: sp2} implies that
$$|A+AA|>\frac{q|A|^3}{2q^2} = \frac12q^{1/2+3\delta}.$$
\end{proof}

Note that the condition that $0\not\in A$ is included only to facilitate the cleanest statement possible. There are a number of comparable sum-product results for large subsets of finite fields; we refer particularly to \cite{garaev} and to \cite{hi, hip}. \footnote{M. Rudnev has pointed out to me that Proposition~\ref{p: sumsets} can be proved in an alternative way, as a
 consequence of a Szemer\'edi-Trotter type theorem (for instance \cite[Theorem 3]{vinh}). The proof goes as follows: for each $x\in A, y\in A+AA$, one defines a line $l_{xy}$ in $(\mathbb{F}_q)^2$ as the set of $(a,b)$ such that $a+bx = y$ (cf. \eqref{e: sp}). Define $\mathcal{L}$ to be the set of all such lines and define $\mathcal{P}$ to be the set $A\times A \subset (\mathbb{F}_q)^2$.
 Observe that the set of incidences of $\mathcal{L}$ with $\mathcal{P}$ is at least $|A|^3$ (since every triple $(a,b,x)\in A^3$ yields a value $y\in A$). Then, since $|\mathcal{L}|=|A|\cdot |A+AA|$ and $|\mathcal{P}|=|A|^2$, \cite[Theorem 3]{vinh} yields the result. Analogous methods yield similar results in the Euclidean plane.}

\subsection{Diameter and width}\label{s: simple}

Our original motivation for this paper was to try and solve two outstanding conjectures in group theory. The first posits an upper bound on the {\it diameter} of a Cayley graph of a finite non-abelian simple group.

\begin{conj}[{\cite[Conjecture~$1.7$]{babaiseress}}]\label{c: babai}
{\rm (Babai's conjecture)} There exists an absolute constant $c$ such that, if $G$ is a finite non-abelian simple group and $S$ is a generating subset of $G$, we have $G=A^k$ where $k\leq (\log |G|)^c$.
\end{conj}

The second posits an upper bound on a {\it width} of a finite non-abelian simple group.

\begin{conj}[{\cite{lns2}}]\label{c: pdc}
 {\rm (The Product Decomposition Conjecture)} There exists an absolute constant $c$ such that if $G$ is a finite non-abelian simple group and $S$ is a subset of $G$ of size at least two, then $G$ is a product of $N$ conjugates of $S$ for some $N\leq c \log|G|/ \log |S|$.
\end{conj}

Both of these conjectures are proved for groups of Lie type of bounded rank \cite{bgt2, ps2, gpss}. We are able to give partial results for groups of Lie type of unbounded rank that complement those already in the literature due to the original Gowers trick.

\begin{prop}\label{p: alternating group}
Fix $\alpha$ a positive real number, let $n$ be odd and let $G=A_n$, the alternating group on $n$ letters. Let $C$ be a conjugacy class of $n$-cycles and suppose that $S\subset G$ such that $S\cap C\neq \emptyset$ and so that
$$|S|\geq \left(\frac{1}{\frac12 n(n-3)}\right)^{1-\alpha}|G|.$$
Then there exists a positive integer $k$, depending only on $\alpha$, such that $G=(S\cup S^{-1})^k$.
\end{prop}

Elements in the conjugacy class $C$ here can be characterised as regular semisimple elements whose centralizer is a ``maximally non-split torus'' (or, in other language, whose centralizer is a {\it Singer cycle}).

\begin{prop}\label{p: sl}
Fix $\alpha$ a positive real number, let $G=SL_n(2)$ and let $C$ be a conjugacy class of elements whose eigenvalues lie in no proper subfield of $\mathbb{F}_{2^n}$. Suppose that $S\subset G$ such that $S\cap C\neq \emptyset$ and so that
$$|S| \geq \left(\frac{3}{(2^n-1)(2^n-4)}\right)^{1-\alpha}|G|.$$ 
Then there exists a positive integer $k$, depending only on $\alpha$, such that $G=(S\cup S^{-1})^k$.
\end{prop}

We emphasise that in neither of these two propositions does the integer $k$ depend on the variable $n$. Notice too that in neither proposition have we needed to assume that $S$ generates $G$ - this fact is implied by the suppositions on $S$. Significantly the lower bound on $|S|$ is not enough to guarantee generation in either case - one needs the extra supposition on the intersection with $C$. This also explains why the lower bounds that we require are weaker than those required by other versions of the ``Gowers trick'' which apply to arbitrary sets in $A_n$ and $SL_n(2)$ \cite{bnp, npy}.

The two propositions (which are proved in \S\ref{s: simple groups}) imply that Babai's conjecture and the Product Decomposition Conjecture hold for the set $S\cup S^{-1}$ and the group $G$ in each case. Indeed \cite[Corollary~2.3]{babai} implies that Babai's conjecture holds for the set $S$ and the group $G$ in both cases.

\subsection{Structure of the paper}

Theorem~\ref{t: main convolution} is proved in \S\ref{s: bnp} using the linear algebra methods of Babai-Nikolov-Pyber. Theorem~\ref{t: main convolution 2} is proved in \S\ref{s: gowers} using the graph-theoretic methods of Gowers. In \S\ref{s: main result} we derive Theorem~\ref{t: main result} from Theorems~\ref{t: main convolution} and \ref{t: main convolution 2}; we also prove Corollary~\ref{c: trick 3}.  Propositions~\ref{p: alternating group} and \ref{p: sl} are proved in \S\ref{s: simple groups}. Finally we conclude with \S\ref{s: further work} in which we discuss possible future directions for research.

\subsection{Acknowledgments}

It is a pleasure to thank Mark Wildon, Ian Short, Jan Saxl, Misha Rudnev, Jeremy Rickard, Laci Pyber, Marty Isaacs, Jack Button and John Britnell (in reverse alphabetical order!) for their generous help with various parts of this paper. In particular the main idea of \S\ref{s: sl} is due to Jan Saxl. 

\section{The first convolution theorem}\label{s: bnp}

This section is devoted to a proof of Theorem~\ref{t: main convolution}. We use the notation established in the introduction without further comment. Note that, in this section, {\bf all matrices are real}.

\subsection{Circulants}

If $E$ is a matrix whose rows (resp. columns) are labelled by elements of a set $X=\{x_1,\dots, x_m\}$ (resp. $Y=\{y_1,\dots, y_n\}$) then we write $E(x_i,y_j)$ (or simply $E(i,j)$) for the entry in matrix $E$ at row $x_i$, column $y_j$ where $x_i\in X, y_j\in Y$.

A matrix $E$ is said to be \emph{biregular} if its row sums are all equal to a constant $s_r(E)$, and its column sums are all equal to a constant $s_c(E)$. Note that the product of biregular matrices (if defined) is biregular, and the quantities $s_r$ and $s_c$ are multiplicative.  

\begin{lem}\label{l: lambda1}\cite[Proposition 5.2]{bnp}
If $E$ is a non-negative biregular $k\times n$ matrix, then
$$\lambda_1(E^TE)=s_r(E)s_c(E)$$
and a corresponding eigenvector is $\mathbf{1}_n=(1,\dots,1)^T$.
\end{lem}

Recall that a {\it $G$-circulant} of a group $G$ is a $|G|$-by-$|G|$ matrix $M$, with rows labelled by elements of $G$ and columns labelled by elements of $G$, and such that 
\begin{equation}\label{e: circulant}
M(g,h) = M(1,g^{-1}h).
\end{equation}
We extend this idea: for a set $\Omega$ on which $G$ acts we define a {\it $G\Omega$-circulant} to be a $|G|$-by-$|\Omega|$ matrix $M$, with rows labelled by elements of $G$ and columns labelled by elements of $\Omega$, and such that 
\begin{equation}\label{e: gc circulant}
M(g,\omega) =  M(1, g^{-1}\omega).
\end{equation}
Observe that a $G$-circulant is simply a $G\Omega$-circulant where we take $\Omega=G$ and consider the regular left action of $G$ on itself.

\begin{lem}\label{l: biregular}
A $G\Omega$-circulant $E$ is biregular, and $s_c(E) = \frac{|G|}{|\Omega|}s_r(E)$.
\end{lem}
\begin{proof}
To see that row sums are constant, observe that, for $g\in G$,
\[
\sum\limits_{\omega\in \Omega}M(g,\omega)=\sum\limits_{\omega\in \Omega}M(1,g^{-1}\omega) =  \sum\limits_{\omega\in g^{-1}(\Omega)}M(1,\omega) = \sum\limits_{\omega\in \Omega}M(1,\omega).
\]
To see that column sums are constant, observe that, for $\omega\in\Omega$,
\[
\sum\limits_{g\in G}M(g,\omega)=\sum\limits_{g\in G}M(1,g^{-1}\omega) = |\stab_G(\omega)| \sum\limits_{\omega\in \Omega}M(1,g) = \frac{|G|}{|\Omega|}s_r(E).
\]
This completes the proof.
\end{proof}

\subsection{Functions}

Let $\Lambda$ be any set and $Z:\Lambda\to\mathbb{R}$ a function. We need some definitions: 

If $Z$ satisfies the property $\sum_{\lambda\in\Lambda}Z(\lambda)=1$ then we call $Z$ a \emph{probability distribution}. The function $Z$ is said to be \emph{concentrated} on the subset $\Xi$ of $\Omega$ if $Z(g)=0$ whenever $g\in \Lambda\setminus \Xi$. We define the norm of $Z$ as the positive square root of $\|Z\|^2= \sum_{\lambda\in \Lambda}Z(\lambda)^2$.

\subsubsection{Convolution}

Consider two functions $X:G\to\mathbb{R}$ and $Y:\Omega\to\mathbb{R}$. At \eqref{e: convolution definition} we defined the notion of \emph{convolution} for $X$ and $Y$, namely:
$$(X\ast_c Y)(\omega)=\sum\limits_{g\in G} X(g)Y(g^{-1}\omega).$$
Observe that
$$\sum\limits_{\omega\in\Omega}(X\ast_c Y)(\omega) = \left(\sum\limits_{g\in G} X(g)\right)\left(\sum\limits_{\omega\in\Omega} Y(\omega)\right).$$
In particular if $X$ and $Y$ are probability distributions then $X\ast_cY$ is also a probability distribution; on the other hand if either $X$ or $Y$ sum to $0$ then $X\ast_c Y$ sums to $0$. 

The key fact about convolutions is this: Suppose that $X:G\to\mathbb{R}$ is concentrated on $S\subset G$, and $Y:\Omega\to\mathbb{R}$ is concentrated on $\Gamma\subset \Omega$; it follows that $(X\ast_c Y)$ is concentrated on $S(\Gamma)$.

\subsubsection{Norms}

We close this section with a number of facts about norms.

\begin{lem}\label{l: 32}
Let $Z$ be a function on $\Omega$ that sums to $0$, $Y$ be a probability distribution over $\Omega$, $X$ a probability distribution over $G$, and $U$ the uniform probability distribution over $\Omega$. Then
\begin{enumerate}
\item $\Vert Z+U\Vert^2 = \Vert Z \Vert^2 + \frac{1}{|\Omega|}.$
 \item $\Vert Y-U\Vert^2 = \Vert Y \Vert^2 - \frac{1}{|\Omega|}.$
\item $\Vert X\ast_c (Y\pm U)\Vert  = \Vert X\ast_c Y\pm U\Vert.$
\item For $k$ a real number $\Vert k Y \Vert = k\Vert Y\Vert.$
\end{enumerate}
\end{lem}
\begin{proof}
For the first fact observe that
\[
\Vert Z+U\Vert^2 = \sum_{\omega\in\Omega} \left(Z(\omega)+\frac{1}{|\Omega|}\right)^2 = \|Z\|^2 +\frac{1}{|\Omega|}+ \frac{2}{|\Omega|}\sum_{\omega\in\Omega}Z(\omega)=\Vert Z \Vert^2 + \frac{1}{|\Omega|}.
\]
For the second fact observe that
\[
\Vert Y-U\Vert^2 = \sum_{\omega\in\Omega} \left(Y(\omega)-\frac{1}{|\Omega|}\right)^2 = \|Y\|^2 +\frac{1}{|\Omega|}- \frac{2}{|\Omega|}\sum_{\omega\in\Omega}Y(\omega)=\Vert Y \Vert^2 - \frac{1}{|\Omega|}.
\]

For the third fact observe that
\begin{eqnarray*}
 \begin{aligned}
\Vert X\ast_c (Y-U)\Vert^2 &= \sum\limits_{\omega\in\Omega} \left(\sum\limits_{g\in G} X(g)(Y(g^{-1}\omega) \pm \frac1{|\Omega|})\right)^2 \\
 &= \sum\limits_{\omega\in\Omega} \left(\sum\limits_{g\in G} X(g)Y(g^{-1}\omega) \pm \sum\limits_{g\in G}X(g) \frac1{|\Omega|}\right)^2 \\
 &=\sum\limits_{\omega\in\Omega} \left(\sum\limits_{g\in G} X(g)Y(g^{-1}\omega) \pm \frac1{|\Omega|}\right)^2 \\
&=\Vert X\ast_cY \pm U \Vert^2.
 \end{aligned}
\end{eqnarray*}

The final fact is immediate.

\end{proof}


\subsection{Functions and circulants}

Let us connect the concepts of the last two subsections. Throughout this subsection we consider functions $X:G\to\mathbb{R}$ and $Y:\Omega\to \mathbb{R}$. We define the {\it $G\Omega$-circulant of $Y$} to be the $G\Omega$-circulant $B$ such that $B(g,\omega)=Y(g^{-1}\omega)$. 

We note a special case of this definition: we consider the natural left regular action of $G$ on itself; in this case $G=\Omega$ and we have a $G\Omega$-circulant $A$ for the function $X:G\to\mathbb{R}$. Now $A$ is actually a $G$-circulant, since it satisfies $A(g,h) = X(g^{-1}h)$ and, so as not to confuse matters, we call $A$ the {\it $G$-circulant of} $Y$.

Observe that if $Y$ is a probability distribution then $s_r(B)=1$, and hence $s_c(B)=|G|/|\Omega|$.

Note the following analogue of \cite[(5.25)]{bnp}.

\begin{lem}\label{l: 525}
Let $B$ be a $G\Omega$-circulant of $Y$. Then
\[
\Vert Y \Vert^2=\frac1{|G|}\trace(BB^T).
\]
\end{lem}
\begin{proof}
\begin{equation*}
 \begin{aligned}
\text{Tr}(BB^T) &= \sum\limits_{g\in G} \left(\sum\limits_{\omega\in \Omega} B(g,\omega) B^T(\omega,g) \right) \\
&= \sum\limits_{g\in G} \left(\sum\limits_{\omega\in \Omega} (B(g,\omega))^2 \right) \\
&= \sum\limits_{g\in G} \left(\sum\limits_{\omega\in \Omega} (B(1,g^{-1}\omega))^2 \right) \\
&=|G|\cdot \Vert Y \Vert^2.
 \end{aligned}
\end{equation*}
\end{proof}

\begin{lem}\label{l: circ}
Let $A$ be the $G$-circulant for $X$, let $B$ be the $G\Omega$-circulant for $Y$, and let $D$ be the $G\Omega$-circulant for $X\ast_c Y$. Then $D=AB$.
\end{lem}
\begin{proof}
Observe that
\begin{align*}
AB(g,\omega) &= \sum_{h\in G} A(g,h)B(h,\omega) \\
&= \sum_{h\in G} X(g^{-1}h)Y(h^{-1}\omega) \\
&= \sum_{h\in G} X(h)Y(h^{-1}g^{-1}\omega) \\
&= (X\ast_c Y)(g^{-1}\omega) \\
&= D(g,\omega),
\end{align*}
as required.
\end{proof}

Lemmas \ref{l: 525} and \ref{l: circ} can be combined to yield an analogue of \cite[Proposition 5.6]{bnp}. 

\begin{prop}\label{p: key}
Let $A$ be the $G$-circulant for $X$, and let $B$ be the $G\Omega$-circulant for $Y$. Then
\[
\Vert X\ast_c Y\Vert^2 = \frac1{|G|}\trace(AB B^TA^T).
\]
\end{prop}

\subsection{Connection with representation dimension}\label{s: connection}

Consider a vector space $\R^{|G|}$ (resp. $\R^{|\Omega|}$); we fix a basis and label each element of the basis with an element of $G$ (resp. $\Omega$). We consider three linear maps as follows.

\subsubsection{A basis for $G\Omega$-circulants}
For $\omega\in\Omega$ define a linear map
$$\rho_\omega: \R^{|G|} \to \R^{|\Omega|}, \, \, g\mapsto g\omega.$$
Representing elements of $\R^{|G|}$ as row vectors, the corresponding matrix representation $B_\omega$ of $\rho_\omega$ (via post-multiplication) is
$$B_\omega(g,\gamma)=\left\{\begin{array}{ll}
              1, & \gamma=g\omega \\
0, &\textnormal{otherwise}.
             \end{array}\right.
$$
Note that if we represent elements of $\R^{|G|}$ as column vectors, then the corresponding matrix representation $\rho_\omega$ (via pre-multiplication) is $B_\omega^T$.

The key fact concerning the matrices $B_\omega$ is this: the $G\Omega$-circulant of a function $X:\Omega\to\mathbb{R}$ lies in the span of the set $\{B_\omega \, \mid \omega \in \Omega\}$.

\subsubsection{The left regular representation}
For $g\in G$ define two linear maps
 $$\tau_g: \mathbb{R}^{|G|}\to \mathbb{R}^{|G|}, \, h \mapsto g^{-1}h;$$
$$\tau^o_g: \R^{|G|}\to \R^{|G|}, \, h\mapsto gh.$$
These two actions correspond to the {\it left regular representation} of $G$ (written as a right (resp. left) action).

We represent elements of $\R^{|G|}$ as row vectors and write $X_g$ for the matrix representation of $\tau_g$ via post-multiplication, so $\tau_g: h\mapsto hX_g$.

On the other hand if we represent elements of $\R^{|G|}$ as column vectors then $X_g$ is also the matrix representation of $\tau^o_g$ via pre-multiplication, so $\tau^o_g: h\mapsto X_gh$.

\subsubsection{The permutation representation}\label{s: perm}
For $g\in G$ define a linear map
$$\sigma_g:  \mathbb{R}^{|\Omega|}\to \mathbb{R}^{|\Omega|}, \, \omega \mapsto g^{-1}\omega.$$
$$\sigma^o_g:  \mathbb{R}^{|\Omega|}\to \mathbb{R}^{|\Omega|}, \, \omega \mapsto g\omega.$$
These two actions correspond to the {\it permutation representations} for $G$ acting on $\Omega$ (written as a right (resp. left) action).

Now we represent elements of $\R^{|\Omega|}$ as row vectors and write $Y_g$ for the matrix representation of $\sigma_g$ via post-multiplication, so $\sigma_g: \omega\mapsto \omega Y_g$.

On the other hand if we represent elements of $\R^{|\Omega|}$ as column vectors then $Y_g$ is also the matrix representation of $\sigma^o_g$ via pre-multiplication, so $\sigma^o_g: \omega\mapsto Y_g\omega$. 

We have already seen the notation $1_H^G$ for the representation $\sigma^o_g$. Note that, since the permutation action associated with $1_H^G$ is transitive, we have $\langle 1_H^G, 1_G\rangle =1$ \cite[(5.15)]{isaacs}.

\subsubsection{Commuting actions}

The following lemma connects the three linear maps we have just defined. The fourth identity will be the one we use directly: it asserts that, for $g\in G$, the matrix $1_H^G(g)$ commutes with matrices of the form $B_\omega^T B_\omega$.

\begin{lem}\label{l: commuting actions}
For all $g\in G$ and $\omega\in\Omega$, the following hold:
\begin{enumerate}
 \item $X_g B_\omega = B_\omega Y_g;$
\item $Y_gB_\omega^T = B_\omega^T X_g;$
\item $X_g B_\omega B_\omega^T = B_\omega B_\omega^TX_g;$
\item $Y_g B_\omega^TB_\omega = B_\omega^TB_\omega Y_g$.
\end{enumerate}
\end{lem}
\begin{proof}
For the first identity, let $x\in G$ be represented as a row vector of length $G$. Then
$$xX_gB_\omega = (g^{-1}x)B_\omega = g^{-1}x\omega.$$
On the other hand
$$xB_\omega Y_g = (x\omega)Y_g = g^{-1}x\omega.$$
The result follows.

For the second identity, let $x\in G$ be represented as a column vector of length $G$. Then
$$Y_gB_\omega^Tx = Y_g(x\omega) = gx\omega.$$
On the other hand
$$B_\omega^T X_gx = B_\omega^T (gx) = gx\omega.$$
The result follows.

Now the first two identities imply that 
$$X_g B_\omega B_\omega^T = B_\omega Y_g B_\omega^T = B_\omega B_\omega^TX_g$$
and the third identity follows. Similarly, for the fourth identity, we have
$$Y_g B_\omega^TB_\omega = B_\omega^TX_g B_\omega = B_\omega^TB_\omega Y_g.$$
\end{proof}

\subsubsection{Symmetric matrices}

Before we proceed to the proof of Theorem~\ref{t: main convolution}, we need a couple of easy results about symmetric matrices. 

Observe first that if $B$ is a real matrix, then $B^TB$ is a symmetric matrix. Recall that every $n$-by-$n$ real symmetric matrix $U$ has $n$ real eigenvalues, counting geometric multiplicities, and we denote them by 
\[
\lambda_1(U) \geq \lambda_2(U)\geq \dots \geq \lambda_n(U).
\]

Furthermore $B^TB$ is positive semidefinite, because  
\[
x^TEx = (x^TB^T)(Bx) = \Vert Bx\Vert^2\geq 0,
\]
which means that all eigenvalues of $BB^T$ are real and non-negative.

In the proof of the next lemma we use $I$ to denote the $n$-by-$n$ identity matrix, for any positive integer $n$.

\begin{lem}\label{l: easy peasy}
Suppose that $B$ is a real matrix. Then $BB^T$ and $B^TB$ have the same non-zero eigenvalues, counting geometric multiplicities.
\end{lem}
\begin{proof}
Given a non-zero real number $\lambda$ we can define a linear map from $\text{ker}(B^TB-\lambda I)$ to $\text{ker}(BB^T-\lambda I)$ by $v\mapsto Bv$. This is well defined, because $BB^T(Bv)=B(B^TB)(v)= B(\lambda v)=\lambda Bv$. It is injective, because if $Bv=0$ then $\lambda v = B^TBv=0$, which means $v=0$. We can also define an injective linear map $v\mapsto B^Tv$ from $\text{ker}(BB^T-\lambda I)$ to $\text{ker}(B^TB-\lambda I)$. Therefore both eigenspaces have the same dimension, as required.
\end{proof}

Note that, in particular, ${\mathrm Tr}(BB^T)={\mathrm Tr}(B^TB)$.

\subsection{Proof of Theorem~\ref{t: main convolution}}

We are just about ready to give a proof of Theorem~\ref{t: main convolution} using the methods of \cite{bnp}. Recall that $X:\Omega\to\mathbb{R}$ and $Y:G\to\mathbb{R}$ are probability distributions; in particular this means that the corresponding circulants are non-negative real matrices. This is crucial in what follows (and will not apply when we come to prove Theorem~\ref{t: main convolution 2}).

In this section we write $U_\Omega$ (resp. $U_G$) for the uniform probability distribution over the set $\Omega$ (resp. over $G$). We begin with an analogue of \cite[Lemma 5.7]{bnp}.

\begin{prop}\label{p: wowsers}
Let $H=\stab_G(\omega)$ and let $d_H$ be the minimum degree of an irreducible component of the representation $1_H^G$. If $B$ is a nonnegative $G\Omega$-circulant, then
\begin{equation}\label{e: wowsers}
 \lambda_2(BB^T) \leq \frac{{\mathrm Tr}(BB^T) - \lambda_1(BB^T)}{d_H}.
\end{equation}
\end{prop}
\begin{proof}
Let $D=BB^T$ and $E=B^TB$. Since $E$ is symmetric and positive semidefinite, all eigenvalues of $E$ are real and non-negative. We denote them by
\[
\lambda_1\geq \lambda_2\geq \cdots \geq \lambda_{|\Omega|}.
\]
Lemma~\ref{l: easy peasy} implies that the eigenvalues of $D$ are  
\[
\lambda_1\geq \lambda_2\geq \cdots \geq \lambda_{|\Omega|}\geq 0 = 0 =\cdots = 0.
\]

Observe that, since $B$ is a $G\Omega$-circulant, it is biregular and so the same is true of $B^T$. Now Lemma~\ref{l: lambda1} implies that $\lambda_1(E)=s_r(B)s_c(B)$, and a corresponding eigenvector is ${\bf 1}=(1,\dots, 1)$.

Observe that the representation $1_H^G$ preserves the one-dimensional subspace spanned by ${\bf 1}$ (since it is a permutation representation). Then, since $\langle 1_H^G, 1_G\rangle =1$, all other subspaces stabilized by $1_H^G$ have non-trivial irreducible components.

Now, since $1_H^G(g)$ commutes with $E$ for every $g\in G$ (this is the fourth identity of Lemma \ref{l: commuting actions}) it follows that all eigenspaces of $E$ are stabilized by $1_H^G$. It follows that the multiplicity of every eigenvalue of the restriction of $E$ to $U$ is at least $d_H$. Lemma \ref{l: easy peasy} implies that the same can be said for the multiplicity of every eigenvalue of the restriction of $D$ to $U$; in particular it is true of the eigenvalue $\lambda_2(D)$. Since the trace of $D$ restricted to $U$ is ${\mathrm Tr}(D) - \lambda_1(D)$ we conclude that
$${\mathrm Tr}(D) - \lambda_1(D) \geq d_H\lambda_2(D).$$ 
\end{proof}

\begin{lem}\label{l: trace}\cite[Lemma 5.8]{bnp}
Let $A$ and $B$ be nonnegative biregular matrices such that the product $AB$ is defined. Then 
\[
\textnormal{Tr}(B^TA^TAB)\leq \lambda_1(A^TA)\lambda_1(B^TB)+\lambda_2(A^TA)(\textnormal{Tr}(B^TB)-\lambda_1(B^TB)).
\]
\end{lem}

We are now ready to prove Theorem~\ref{t: main convolution}.

\begin{proof}[Proof of Theorem~\ref{t: main convolution}]
Let $A$ be the $G$-circulant for $X$ and $B$ be the $G\Omega$-circulant for $Y$. Proposition~\ref{p: key} and Lemma~\ref{l: trace} imply that
\[
 \begin{aligned}
  \Vert X\ast_c Y\Vert^2 &= \frac{1}{|G|}{\mathrm Tr}(AB B^TA^T) \\
&\leq \frac{1}{|G|} \lambda_1(BB^T)\lambda_1(AA^T) + \frac{1}{|G|}\lambda_2(BB^T)(\trace(AA^T) - \lambda_1(AA^T)).
 \end{aligned}
\]
Because $AA^T$ is nonnegative and biregular we see that $\lambda_1(AA^T)= s_r(A)s_c(A) = 1$, and using Lemma~\ref{l: biregular} we see that $\lambda_1(BB^T)=s_r(B)s_c(B)=|G|/|\Omega|.$ Using Lemma~\ref{l: 32} it follows that
\[
\Vert X\ast_c Y-U_\Omega\Vert^2 = \Vert X\ast_c Y\Vert^2 -\frac{1}{|\Omega|} \leq  \frac{1}{|G|}\lambda_2(BB^T)(\trace(AA^T) - 1).
\]
Therefore, by Proposition~\ref{p: wowsers},
\[
 \begin{aligned}
  \Vert X\ast_c Y-U_\Omega\Vert^2 &\leq  \frac{1}{|G|d_H}\left(\trace(BB^T) - \frac{|G|}{|\Omega|}\right)\left(\trace(AA^T) - 1\right) \\ 
&= \frac{|G|}{d_H}\left(\Vert Y\Vert^2 - \frac{1}{|\Omega|}\right) \left(\Vert X \Vert^2 - \frac{1}{|G|}\right) \\
&= \frac{|G|}{d_H} \Vert Y-U_\Omega\Vert^2 \Vert X - U_G \Vert^2.
\end{aligned}
\]
\end{proof}

\section{The second convolution theorem}\label{s: gowers}

In this section we prove Theorem~\ref{t: main convolution 2} using the methods of Gowers \cite{gowers}. Although much of Gowers' work can be reframed without referring to his original graph-theoretic setting this would seem to be a mistake: it is difficult to retain intuition about what is going on once one has ``linearized'' and written everything in terms of matrices. On the other hand the geometry of the group action is nicely encapsulated by the graphs that Gowers considers and so we make use of them here.

\subsubsection{Bipartite graphs}

In what follows $\graphg$ is a bipartite graph with vertex sets $X$ and $Y$. We write $\mathcal{A}$ for the adjacency matrix of $\graphg$. Note that, unlike for Gowers, our graph $\graphg$ is not necessarily simple, i.e. we allow the possibility that there is more than one edge between two vertices. This implies, in particular, that the entries of $\mathcal{A}$ may exceed 1.

Our first job is to analyse $\mathcal{A}$ and for this we will need some notation given on \cite[p. 7]{gowers}. We let $V$ and $W$ be real vector spaces with the usual inner product. For $v\in V, w\in W$ define the linear map
$$w\otimes v: V \to W, x\mapsto \langle x, v \rangle w.$$
We need the following result:

\begin{prop}\label{p: 26}\cite[Theorem 2.6]{gowers}
Let $\alpha: V\to W$ be a linear map. Then there exists a decomposition $\alpha=\sum_{i=1}^k \lambda_i w_i\otimes v_i$ where the sequences $(w_i)$ (resp. $(v_i)$) are orthonormal in $W$ (resp. $V$), the sequence $(\lambda_i)$ is real, non-negative and non-increasing, and $k=\min\{\dim V, \dim W\}$. 

Note, in addition, that the sequence $(\lambda_i)$ is uniquely determined, and that the vector $v_1$ can be taken to be any vector such that, for all $v\in V$, 
$$\|\alpha(v_1)\|/\|v_1\| \geq \|\alpha(v)\|/\|v\|$$
\end{prop}

The last sentence of Proposition~\ref{p: 26} does not appear in the statement of \cite[Theorem 2.6]{gowers} but is clear from the proof.

Our next result is an analogue of \cite[Lemma 2.7]{gowers} adjusted to hold for graphs which are not simple; in fact we will only need part of the original lemma.

\begin{lem}\label{l: 27}
 Let $\graphg$ be a bipartite graph with vertex sets $X$ and $Y$ and identify $\graphg$ with its bipartite adjacency matrix $\sum_{i=1}^k \lambda_i w_i\otimes v_i$, where $(v_i)$ and $(w_i)$ are orthonormal sequences. Then the number of edges in $\graphg$ is greater than or equal to 
$$\frac{1}{\ell}\sum_{i=1}^k\lambda_i^2$$
where $\ell$ is the maximum number of edges between any two vertices of $\graphg$.
\end{lem}
\begin{proof}
Observe first that $\mathcal{A}^T$ is $\sum_i\lambda_i v_i\otimes w_i$ and that $$(v_i\otimes w_i)(w_j\otimes v_j)=
\begin{cases}
v_i\otimes v_i, &i=j, \\
0, & \textrm{otherwise}.
\end{cases}$$
Now ${\trace}(v_i\otimes v_i)=1$ and thus ${\trace}(\mathcal{A}^T\mathcal{A})=\sum_i\lambda_i^2$.

But now $\frac{1}{\ell}{\trace}(\mathcal{A}^T\mathcal{A})$ is less than or equal to the number of edges in $\graphg$.
\end{proof}

We use the graph $\graphg$ to define the following map:
$$\alpha: \mathbb{R}^X \to \mathbb{R}^Y, \, \, f \mapsto \alpha f$$
where, for $f: X\to \mathbb{R}$ we have
\begin{equation}\label{e: alpha 1}
\alpha f: Y\to \mathbb{R}, \, \, y\mapsto \sum_{x\in X, xy\in E(\graphg)} f(x).
\end{equation}

Note that, if there is more than one edge between two vertices $x$ and $y$, then our definition of $(\alpha)(f)(y)$ requires that the value $f(x)$ is added multiple times - once for each edge between $x$ and $y$.

The map $\alpha$ will be central in what follows and we shall see in the next subsection that it is closely related to the idea of convolution. 

The following lemma contains everything that we need to know about the map $\alpha$. In the statement of the lemma, the graph $\graphg$ is assumed to be {\it regular}, i.e. every vertex in $X$ has the same degree and every vertex in $Y$ has the same degree.  We set $\lambda_i, v_i, w_i$ and $k$ to be as defined in Proposition~\ref{p: 26}.


\begin{lem}\label{l: g29}
Suppose that $\graphg$ is a regular bipartite graph. The following hold.
\begin{enumerate}
\item $\lambda_1=\max \{\|\alpha(f)\|/ \|f\|\, \mid \, f\in \mathbb{R}^\Omega\} = \|\alpha(v_1)\|/\|v_1\|$.
\item We can take $v_1$ to be the constant function
\begin{equation}\label{e: 1}
X\to \mathbb{R}, \, x \mapsto \frac{1}{\sqrt{|X|}}.
\end{equation}
\item The set $\mathcal{F}$ of functions $X\to \mathbb{R}$ that sum to zero is a vector space of dimension $k-1$. 
\item For all $f\in \mathcal{F}$, $\|\alpha(f)\|/ \|f\|\leq \lambda_2$.
\item Let $e$ be the positive integer such that
$$\lambda_2=\lambda_3=\cdots=\lambda_e>\lambda_{e+1}.$$
Then the set $\mathcal{E}$ of functions in $\mathcal{F}$ such that $\|\alpha(f)\|/ \|f\|=\lambda_2$ is a vector space (provided we include $0$) of dimension $e$.
\end{enumerate}
\end{lem}
\begin{proof}
Observe first that 
\begin{eqnarray}\label{e: alpha formula}
\alpha\left(\sum\limits_{i=1}^k \mu_i v_i\right) = \sum\limits_{i=1}^k \lambda_i\mu_i w_k.
\end{eqnarray}
In particular (1) holds.

To prove (2) we set $p$ to be the real number such that every vertex in $X$ has degree $p|Y|$; observe that, since $\graphg$ is regular, every vertex in $Y$ has degree $p|X|$. Now
\begin{eqnarray*}
 \begin{aligned}
\|\alpha f\|^2 &= \sum\limits_y\left|\sum\limits_x f(x)\mathcal{A}(x,y)\right|^2 \\
  &= \sum\limits_{x,x'} f(x)f(x')\sum\limits_y\mathcal{A}(x,y)\mathcal{A}(x',y) \\
&\leq \frac12 \sum\limits_{x,x'} \left(f(x)^2+f(x')^2\right)\sum\limits_y\mathcal{A}(x,y)\mathcal{A}(x',y) \\
&=\sum\limits_x f(x)^2 \sum\limits_{x'} \sum\limits_y \mathcal{A}(x,y)\mathcal{A}(x',y) \\
&= \sum\limits_x f(x)^2 p^2|X||Y|=p^2|X||Y|\|f\|^2.
 \end{aligned}
\end{eqnarray*}
It follows that $\|\alpha(f)\|/ \|f\|\leq p\sqrt{|X|\cdot |Y|}$. Now let $f=1$, the function defined at \eqref{e: 1} and we have $\|\alpha f\| = p|X|\sqrt{|Y|}$ and $\|f\| = \sqrt{|X|}$ in which case $\|\alpha(f)\|/ \|f\|= p\sqrt{|X|\cdot |Y|}$ and (2) follows. 

Item (3) is immediate once we observe that $\mathcal{F}$ is the orthogonal complement of the function \eqref{e: 1}. Taking $v_1$ to be this function (by (2)) we conclude that $\mathcal{F}$ is spanned by $\{v_2,\dots, v_k\}$ and the map $\alpha|_{\mathcal{F}}$ can be decomposed as $\sum_{i=2}^k\lambda_i w_i\otimes v_i$. Then (4) follows by applying (1) to this decomposition.

Applying (1) to the vector space $\mathcal{F}$ we observe that $\|\alpha(f)\|/ \|f\|\leq \lambda_2$ for all $f\in\mathcal{E}$. Furthermore \eqref{e: alpha formula} implies that $\|\alpha(f)\|/\|f\|=\lambda_2$ if and only if $f$ is in the span of $\{v_2,\dots, v_e\}$. Now (5) is immediate.
\end{proof}

\subsubsection{Graphs from groups}\label{s: graphs groups}

We return to the setting where $G$ is a group acting transitively on a set $\Omega$ and $S$ is a subset of $G$. We will work with the following bipartite graph, $\graphg$: the two vertex sets, $X$ and $Y$, are copies of $\Omega$ and $xy$ is an edge if and only if there exists $s\in S$ such that $s(x)=y$. Note that this graph is {\it regular}, i.e. every vertex in $X$ has the same degree and every vertex in $Y$ has the same degree. 

As before we write $\mathcal{A}$ for the adjacency matrix of $\graphg$. Observe that, for $x,y\in \Omega$, $\mathcal{A}(x,y)$ is the number of edges from $x$ to $y$ in $\graphg$). 

If $S$ is a subset of $G$ we write $\chi_S$ for the characteristic function of $S$. Now, for this particular graph $\graphg$, we can use the more general definition of convolution given at \eqref{e: convolution definition} to describe the function $\alpha$ defined at \eqref{e: alpha 1} in a different way:
\begin{equation}\label{e: alpha}
\alpha f(\omega) = \sum\limits_{\nu\in \Omega} \mathcal{A}(\nu,\omega)f(\nu) = \sum\limits_{g\in G} \chi_S(g)f(g^{-1}\omega) = (\chi_S\ast_c f)(\omega).
\end{equation}
In other words $\alpha(f) = \chi_S\ast_c f$.

Note that the linear function $\alpha: \mathbb{R}^\Omega \to \mathbb{R}^\Omega$ has associated matrix $\mathcal{A}^T$. Note, moreover, that
$$\mathcal{A} = \sum\limits_{g\in S} Y_{g^{-1}}$$
where, for $g\in G$, the matrix $Y_g$ was defined in \S\ref{s: perm}. With these observations in mind we are ready to prove Theorem~\ref{t: main convolution 2}. This is the analogue of \cite[Lemma 3.2]{gowers} and, in Gowers' language, asserts that the graph $\graphg$ is quasirandom.

\begin{proof}[Proof of Theorem~\ref{t: main convolution 2}]
Let $\graphg$ be the bipartite Cayley graph defined above and observe that $\ell_S$ is equal to the maximum number of edges between vertices in $\graphg$. Observe too that $\graphg$ is regular and let $\alpha$ be the associated linear map \eqref{e: alpha}.

By the observations above, the associated matrix for $\alpha$ (once we fix a basis) is equal to $\sum\limits_{g\in S} Y_{g^{-1}}^T$. Since the matrices $Y_{g^{-1}}$ correspond to the permutation representation $1_H^G$, these matrices then preserve a decomposition of $\mathbb{R}^\Omega$ into subspaces, one for each irreducible component of the representation $1_H^G$. Then the vectors $v_1, \dots, v_{|\Omega|}$ can be chosen to lie inside these subspaces. 

Suppose that the vector $v_i$ lies inside a subspace $W$ corresponding to an irreducible component $\chi$ of $1_H^G$. It is easy to see that the corresponding real number $\lambda_i$ will occur in the sequence $(\lambda_1, \dots, \lambda_{|\Omega|})$ with multiplicity at least the multiplicity of the irreducible component $\chi$, this multiplicity being $\langle \chi, 1_H^g\rangle\geq m_H$. 

Let $\mathcal{E}$ and $\mathcal{F}$ be the vector spaces defined in Lemma~\ref{l: g29}. Referring to item (1) of that lemma we take $v_1$ to be the constant function \eqref{e: 1}. The subspace $\langle v_1 \rangle$ is preserved by the matrices $Y_{g^{-1}}$, as is $\mathcal{F}$, the orthogonal complement of $v_1$. Moreover, since $\langle 1_H^G, 1_G\rangle =1$, the subspace $\langle v_1 \rangle$ is the unique 1-dimensional subspace of $\mathbb{R}^\Omega$ that is stabilized by $Y_{g}$ for all $g\in G$. Hence, in particular, all of the subspaces of $\mathcal{F}$ stabilized by the matrices $Y_g$ correspond to irreducible components of $1_H^G$ with multiplicity at least $m_H$. We conclude that the vector space $\mathcal{E}$ must have dimension at least $m_H$, i.e. that the real number $\lambda_2$ occurs with multiplicity at least $m_H$.

Lemma \ref{l: 27} implies that $\frac{1}{\ell} m_H\lambda_2^2$ is less than or equal to the number of edges in $\graphg$. But $\graphg$ has $|S|\cdot |\Omega|$ edges and we conclude that
\begin{equation}\label{e: wool}
\lambda_2 \leq \sqrt{\ell_S |\Omega|/m_H} \cdot \sqrt{|S|}.
\end{equation}
Lemma~\ref{l: g29} part (4) implies that if $f:X\to \mathbb{R}$ is a function that sums to zero, then $\|(\alpha f)\|/\|f\| \leq \lambda_2$. Observing that $\|\chi_S\|=\sqrt{|S|}$ and substituting into \eqref{e: wool} we obtain
$$\|\alpha f\|/\|f\| \leq \lambda_2 \leq \sqrt{\ell_S |\Omega|/m_H} \cdot\|\chi_S\|.$$
Now \eqref{e: alpha} gives the result.
\end{proof}

\section{Large sets grow}\label{s: main result}


\begin{proof}[Proof of Theorem~\ref{t: main result}]
Let $X$ be the probability distribution over $G$, $Y$ the probability distribution over $\Omega$ given by the following definitions: 
\[
X(x) = 
\begin{cases}
\frac{1}{|S|}, & x\in S, \\
0, & x\notin S,
\end{cases}
\qquad          
Y(x) = 
\begin{cases}
\frac{1}{|\Gamma|}, & x\in \Gamma, \\
0, & x\notin \Gamma.
\end{cases}
\]          
Observe that $\Vert X\Vert = \frac1{\sqrt{|S|}}$ and $\Vert Y\Vert = \frac1{\sqrt{|\Gamma|}}$. Recall that $X\ast_c Y$ is {\it concentrated} on $S(\Gamma)$, meaning that $(X\ast_c Y)(g)=0$ whenever $g\in \Omega\setminus S(\Gamma)$. A simple application of the Cauchy-Schwarz inequality (or see \cite[Observation 3.4]{bnp}) gives
\begin{equation}\label{e: cs}
\frac{1}{|S(\Gamma)|} \leq \Vert X\ast_c Y \Vert^2. 
\end{equation}
This inequality, with Lemma~\ref{l: 32} and Theorem~\ref{t: main convolution}, imply that
\begin{equation*}
 \begin{aligned}
  \frac{1}{|S(\Gamma)|} &\leq \Vert X\ast_c Y \Vert^2 \\
&\leq \frac{1}{|\Omega|}+ \Vert X\ast_c Y - U_\Omega \Vert^2 \\
&\leq \frac{1}{|\Omega|}+ \frac{|G|}{d_H}\Vert X-U_G\Vert^2 \Vert Y-U_\Omega\Vert^2 \\
&< \frac{1}{|\Omega|} + \frac{|G|}{d_H} \Vert X\Vert^2 \Vert Y\Vert^2 \\
&= \frac{1}{|\Omega|} + \frac{|G|}{d_H}\frac1{|S|}\frac1{|\Gamma|}.
 \end{aligned}
\end{equation*}
Rearranging we obtain  
\[
|S(\Gamma)| > \frac{|\Omega|}{1+\frac{|G||\Omega|}{d_H|S||\Gamma|}},
\]
which is the first inequality of \eqref{e: main sets}. For the second inequality, observe that if $\frac{|G||\Omega|}{d_H|S||\Gamma|}\leq 1$ then 
\[
\frac{|\Omega|}{1+\frac{|G||\Omega|}{d_H|S||\Gamma|}} \geq \frac{|\Omega|}{2}.
\]
On the other hand, if $\frac{|G||\Omega|}{d_H|S||\Gamma|}> 1$ then 
\[
\frac{|\Omega|}{1+\frac{|G||\Omega|}{d_H|S||\Gamma|}} = \frac{|\Omega|d_H|S||\Gamma|}{d_H|S||\Gamma|+|G||\Omega|}> \frac{|\Omega|d_H|S||\Gamma|}{2|G||\Omega|}=\frac{d_H|S||\Gamma|}{2|G|}.
\]
In both cases, the second inequality holds.

Now we must prove \eqref{e: main sets 2}. We begin by defining $U_\Omega$ to be the uniform probability distribution over $\Omega$ and observe that $f=Y - U_\Omega$ is a function on $\Omega$ that sums to $0$. Observe too that $\chi_S = |S| p_S$.

Now we start with \eqref{e: cs}, apply Theorem~\ref{t: main convolution 2} and make use of the identities in Lemma~\ref{l: 32}:
\begin{equation*}
 \begin{aligned}
  \frac{1}{|S(\Gamma)|} &\leq \Vert X\ast_c Y \Vert^2 \\
&=\| X\ast_c (f+U_\Omega)\|^2 \\
&=\|X \ast_c f + U_\Omega\|^2 \\
&=\|X \ast_c f\|^2 + \frac1{|\Omega|}\\
&=\frac1{|S|^2}\|\chi_S\ast_c f\|^2 + \frac1{|\Omega|}\\
&\leq \frac1{|S|^2}\cdot\frac{\ell_S|\Omega|}{m_H} \Vert \chi_S\Vert^2 \Vert f\Vert^2 +
\frac{1}{|\Omega|}\\
&= \frac1{|S|^2}\cdot\frac{\ell_S|\Omega|}{m_H} \cdot |S| \cdot \Vert p_\Gamma-U_\Omega \Vert^2 +\frac{1}{|\Omega|}\\
&< \frac{\ell_S|\Omega|}{m_H}\frac1{|S|}\frac1{|\Gamma|}+\frac{1}{|\Omega|}.
 \end{aligned}
\end{equation*}
Rearranging we obtain  
\[
|S(\Gamma)| > \frac{|\Omega|}{1+\frac{\ell_S|\Omega|^2}{m_H|S||\Gamma|}},
\]
which is the first inequality of \eqref{e: main sets 2}. The second inequality follows just as for \eqref{e: main sets}.

\end{proof}

\subsection{Corollary~\ref{c: trick 3}}\label{s: corollaries}

In this subsection we prove Corollary~\ref{c: trick 3}. By way of introduction we state a weaker result, the proof of which illustrates our methods.

\begin{cor}\label{c: trick 2}
Let $G$ be a finite group and let $C$ be a conjugacy class of $G$. Let $H$ be the centralizer of an element of $C$ and let $A$ be a subset of $C$. Suppose that
\begin{enumerate}
 \item $|A|\geq \frac{|C|}{2}$ and
\item $d_H > 8|H|\ell_C$,
\end{enumerate}
Then $$(A\cup A^{-1})^{5}\supseteq AAAA^{-1}A^{-1}\supseteq C.$$
\end{cor}
\begin{proof}
Write $n$ for $|C|$. We apply Corollary~\ref{c: quasirandom main} with $S=\Delta_1 = A$ and $\Delta_2$ the set of elements that are {\bf not} in the set $S(\Delta_1)$ i.e. are not of the form $a_1 a_2 a_1^{-1}$ for some $a_1, a_2 \in A_1$. We use the fact that $\ell_S = \ell_A\leq \ell_C$ and obtain that
$$|\Delta_2| \leq \frac{n^2|G|\ell_C}{m_H}/ (\frac{n}2)^2 = \frac{4|G|}{d_H}.$$
Thus the set $A_2 = \bigcup\limits_{a\in A} aAa^{-1}$ has size at least 
\begin{equation}\label{e: a2}
 n-|\Delta_2|\geq n-\frac{4|G|}{d_H}.
\end{equation}

Now, for $g\in C$, define $B_g=\{a^{-1}ga \mid a\in A_1\}$ and observe that 
$$|B_g|\geq \frac{|A|}{\ell_C} \geq \frac{n}{2\ell_C}.$$
Now, since $e_H > 8|H|\ell_C$, a little rearranging yields that
$$\frac{n}{2\ell_C}>\frac{4|G|}{d_H}.$$
Thus, by the pigeonhole principle $B_g \cap A_2$ is non-empty for every $g\in C$. We conclude, therefore, that
$$A_3 = \bigcup\limits_{a\in A} aA_2a^{-1} = C.$$
Now 
$$A_3 \subseteq AA_2A^{-1} \subseteq AAAA^{-1}A^{-1}$$
and the result follows. 
\end{proof}

It turns out that the bound (2) needed for Corollary~\ref{c: trick 2} is too strong for wide application, hence the need for the stronger statement given in Corollary~\ref{c: trick 3}.

\begin{proof}[Proof of Corollary~\ref{c: trick 3}]
We define the sets $A_2$ and $B_g$ as per the previous proof, and we recall \eqref{e: a2}:
$$|(A\cup A^{-1})^3 \cap C|\geq |A_2| \geq |C|-\frac{4|G|}{d_H}.$$
Using the fact that $m_H > \frac{8}{k}|H|\ell_C$ we observe that
$$|B_g|\geq \frac{|A|}{\ell_C} \geq \frac{|C|}{2\ell_C}>\frac{4|G|}{k\cdot d_H}.$$

The first step of our proof involves building a set $X$ with particular properties; we begin by setting $X=\emptyset$.  Now suppose that, for all $g_1\in C\backslash A_2$, we have 
$$B_{g_1}\cap \left(A_2 \cup \bigcup\limits_{g\in X} B_g\right) = \emptyset.$$
In this case we add $g_1$ to our set $X$ and repeat. Since $|B_g|\geq \frac{4|G|}{k\cdot d_H}$ we can repeat this process until $X$ has size at most $k$, at which point no such $g_1$ will exist. In this case we stop.

By way of comparison with the previous result note that if $X=\emptyset$ then we obtain immediately that $B_g\cap A_2\neq \emptyset$ for every $g\in C$ and we obtain, as required that 
$$(A\cup A^{-1})^{5}\supseteq AA_2A^{-1}\supseteq C.$$

If $X$ is not empty, we have a little more work to do. Observe first that
$$A_2\cup \bigcup\limits_{g\in X}AB_gA^{-1} \supseteq C.$$
Now $AA_2A^{-1}$ is strictly larger than $A_2$ and hence intersects $AB_gA^{-1}$ for some $g\in X$. Thus $A^{-1}AA_2A^{-1}A$ intersects $B_g$ and thus
$$g\in AA^{-1}AA_2A^{-1}AA^{-1}.$$
Then $B_g\subset A^{-1}AA^{-1}AA_2A^{-1}AA^{-1}A$ and, finally,
$$AB_gA^{-1}\subseteq AA^{-1}AA^{-1}AA_2A^{-1}AA^{-1}AA^{-1}.$$
Since $A_2\subseteq (A\cup A^{-1})^{3}$ we obtain that 
$$AB_gA^{-1}\subseteq (A\cup A^{-1})^{13}.$$

Now we repeat the process with $A_2$ redefined to be $(A\cup A^{-1})^{13}\cap C$. We can repeat this at most $k$ times at the end of which $A_2$ is the set $(A\cup A^{-1})^{3+10k}\cap C$ and it has the property that 
$B_g\cap A_2\neq \emptyset$ for every $g\in C$. Now we conclude, as in the previous proof, that
$$(A\cup A^{-1})^{5+10k}\supseteq AA_2A^{-1}\supseteq C.$$
\end{proof}


\section{Simple groups}\label{s: simple groups}

In this section we prove Propositions~\ref{p: alternating group} and \ref{p: sl}. We need a lemma.

\begin{lem}\label{l: simples}
Let $G$ be a finite group and let $C$ be a conjugacy class of $G$. Let $H$ be the centralizer of an element of $C$ and let $S$ be a subset of $G$. Suppose that there exists a positive number $\alpha$ such that
\begin{enumerate}
 \item $|S|\geq \left(\frac{1}{d_H}\right)^{1-\alpha}|C|$.
\item $|S\cap C| \geq \left(\frac{1}{d_H}\right)^{3}|C|.$
\end{enumerate}
Then $|(S\cup S^{-1})^{2\lceil\frac{3}{\alpha}\rceil - 1} \cap C| \geq \frac{|C|}{2}$.
\end{lem}
\begin{proof}
Applying Theorem~\ref{t: main result} with $\Gamma = S\cap C$ we conclude that
$$(S\cup S^{-1})^3 \geq \frac12 \min\{|C|, (d_h)^{\alpha-3}|C|\}.$$
Iterating we conclude that, for $k$ a positive integer,
$$(S\cup S^{-1})^{2k-1} \geq \frac12 \min\{|C|, (d_h)^{k\alpha-3}|C|\}.$$
Taking $k=\lceil \frac{3}{\alpha}\rceil$ the result follows.
\end{proof}

It will be convenient to use the following result of Liebeck and Shalev \cite{lieshal}\footnote{This is a spectacular sledgehammer to crack a couple of rather tiny nuts. Nonetheless it saves us some tiresome computations in each case.} Note that a {\it normal subset} of a group  is a union of conjugacy classes.

\begin{thm}\label{t: normal}
There exists an absolute positive constant $a$ such that, if $G$ is a finite simple group and $S$ is a nontrivial normal subset of $G$, then $G=S^m$, where $m\leq a\frac{\log|G|}{\log|S|}$.
\end{thm}

\subsection{Alternating groups}\label{s: alternating group}

This section is devoted to a proof of Proposition~\ref{p: alternating group}. We write representations of $S_n$ in the standard way: indexed by partitions of $n$. Then \cite{rasala} implies:

\begin{lem}\label{l: min degree sym}
Suppose that $n\geq 15$. The first seven minimal character degrees $d$ of $S_n$ are given by representations $S^\lambda$ as follows:
\begin{enumerate}
 \item $d=1$ and $\lambda\in\{ (n), (1^n)\}$;
\item $d=n-1$ and $\lambda\in\{(n-1,1), (2, 1^{n-2})\}$;
\item $d=\frac12 n(n-3)$ and $\lambda\in\{(n-2,2), (2,2,1^{n-4})\}$;
\item $d=\frac12(n-1)(n-2)$ and $\lambda\in\{(n-2,1,1), (3,1^{n-3}) \}$;
\item $d=\frac16n(n-1)(n-5)$ and $\lambda\in\{(n-3,3), (2,2,2,1^{n-6})\}$;
\item $d=\frac16(n-1)(n-2)(n-3)$ and $\lambda\in\{(n-3,1^3), (4, 1^{n-4})\}$;
\item $d=\frac13n(n-2)(n-4)$ and $\lambda\in\{(n-3,2,1), (3,2,1^{n-5})\}$.
\end{enumerate}
\end{lem}

Note that there are two representations in each case; they correspond to tensoring by the sign representation. In terms of partitions they correspond to reflecting in the diagonal.

Note that, for $n\geq 15$, the two partitions listed in each case are distinct, i.e. tensoring by the sign representation yields a non-isomorphic representation. It follows by \cite[Prop. 5.1]{fulhar} that the representations given in Lem. \ref{l: min degree sym} stay irreducible when restricted to $A_n$. Furthermore any irreducible representation of $A_n$ is obtained by restriction from an irreducible representation of $S_n$ and it will either have the same degree (as in the above cases for $n\geq 15$) or half the original degree. Since $\frac13n(n-2)(n-4)$ is more than double $\frac12n(n-3)$ for $n\geq 15$ we conclude the following:

\begin{lem}\label{l: min degree alt}
 Suppose that $n\geq 15$. The first three minimal character degrees $d$ of $A_n$ are given by representations $S^\lambda$ as follows:
\begin{enumerate}
 \item $d=1$ and $\lambda\in\{ (n), (1^n)\}$;
\item $d=n-1$ and $\lambda\in\{(n-1,1), (2, 1^{n-2})\}$;
\item $d=\frac12 n(n-3)$ and $\lambda\in\{(n-2,2), (2,2,1^{n-4})\}$.
\end{enumerate}
\end{lem}

Note that although, again, there are two partitions listed in this case, the corresponding representations of $A_n$ are isomorphic to each other. Thus there is really only one representation of $A_n$ of the given degree.

Let $H$ be a subgroup of $G=A_n$ and observe that $ (1_H^G)_G^{S_n} = 1_H^{S_n}.$ Consider $\theta$ to be a character of $S_n$. Then Frobenius reciprocity implies that
\begin{equation}\label{e:orbits2}
\langle 1_H^G, \theta|_G \rangle= \langle  (1_H^G)_G^{S_n}, \theta\rangle = \langle 1_H^{S_n} , \theta\rangle 
\end{equation}
where $\theta_G$ is the restriction of $\theta$ to $G$.

If $n\geq 15$ and $\theta$ is one of the characters of $S_n$ associated with the representations $S^{(n)}, S^{(n-1,1)}, S^{(n-2,2)}$, then $\theta$ corresponds to a partition which is not symmetric through the diagonal. Thus \cite[Prop. 5.1]{fulhar} implies that $\theta_G$ is one of the three minimal characters listed in Lemma \ref{l: min degree alt}.

We need a result of Frobenius. From here on, for a partition $\lambda$ of $n$ we write $\theta^\lambda$ for the character of $S_n$ associated with the representation $S^\lambda$; similarly we write $\chi^\lambda$ for the character of $A_n$ associated with $S^\lambda$.

\begin{lem}\label{l: frobenius}
Let $H\leq S_n$ be a permutation group acting on $\{1, 2, \dots,n\}$. Let $t_r(H)$ be the number of orbits of $H$ on $r$-subsets of $\{1, 2,\dots, n\}$. If $0\leq r\leq n/2$ then
\begin{equation}\label{e:orbits}
 \langle 1_H^{S_n}, \theta^{(n-r,r)}\rangle = \left\{\begin{array}{ll}
                                                    t_r(H) - t_{r-1}(H), &r\geq 1; \\
t_0(H)=1, &r=0.
                                                   \end{array}\right.
\end{equation}
\end{lem}

Now (\ref{e:orbits2}) implies an immediate corollary of Lemma \ref{l: frobenius}:
\begin{lem}\label{l: frobenius2}
Let $H\leq A_n$ be a permutation group acting on $\{1, 2, \dots,n\}$. Let $t_r(H)$ be the number of orbits of $H$ on $r$-subsets of $\{1, 2,\dots, n\}$. If $0\leq r\leq n/2$ then
\begin{equation}\label{e:orbits3}
 \langle 1_H^G , \chi^{(n-r,r)}\rangle = \left\{\begin{array}{ll}
                                                    t_r(H) - t_{r-1}(H), &r\geq 1; \\
t_0(H)=1, &r=0.
                                                   \end{array}\right.
\end{equation}
\end{lem}

Recall that a permutation group $H$ on $\{1,\dots, n\}$ is called $r$-homogeneous (for $r>1$ an integer) if $H$ is transitive on the $r$-subsets of $\{1,\dots, n\}$.
We need one more lemma:

\begin{lem}\label{l: final}
Let $n$ be odd, $G=A_n$ and $g$ an $n$-cycle in $G$. Let $H$ be transitive on $\{1,\dots, n\}$ but not $2$-homogenous. The minimum degree of a non-trivial irreducible component of $1_H^G$ is equal to $\frac12n(n-3)$. 
\end{lem}
\begin{proof}
By Lemmas \ref{l: min degree alt} and \ref{l: frobenius2} we must show that $t_1(H)=1$ and $t_2(H)\geq 1$. That $t_1(H)=1$ follows from the fact that $H$ is transitive; since $H$ is not $2$-homogeneous we conclude that $t_2(H)\geq 1$. 
\end{proof}

If $H$ is the centralizer of an $n$-cycle in $G$, then $H$ is transitive but not $2$-homogeneous. Thus we have the following:

\begin{cor}\label{c: final}
If $H$ is the centralizer of an $n$-cycle in $G$, then $d_H=\frac12 n(n-3)$.
\end{cor}

We can now prove Proposition \ref{p: alternating group}.

\begin{proof}
Note first that the result is trivial for $n$ less than any absolute constant. It will suit us to assume from here on that $n>100$. Set $H=C_G(g)$ and observe that Corollary~\ref{c: final} implies that $d_H=\frac 12 n(n-3)$. 
Since $|H|=n$, we apply the pigeonhole principle to the cosets of $H$ to conclude that
$$|(S\cup S^{-1})^3| \geq \left|\bigcup\limits_{s,g\in S} sgs^{-1}\right| \geq \frac{1}{n}\left(\frac{1}{d_H}\right)^{1-\alpha}|C| > \frac1n\left(\frac{1}{d_H}\right)|C|>\left(\frac{1}{d_H}\right)^3|C|.$$

Thus we can apply Lemma~\ref{l: simples} to the set $(S\cup S^{-1})^3$ to conclude that 
$$|(S\cup S^{-1})^{6\lceil\frac{3}{\alpha}\rceil - 1} \cap C| \geq \frac{|C|}{2}.$$
Let $A_1$ be the set $(S\cup S^{-1})^{6\lceil\frac{2}{\alpha}\rceil-1} \cap C$ and observe that, since $n\geq 100$, we have $m_c>\frac{8}{20}|H|^2\geq \frac{8}{20}\ell_C|H|$. Then Corollary~\ref{c: trick 3}, applied to $A_1$ with $k=20$, implies that $(S\cup S^{-1})^{1230\lceil\frac{2}{\alpha}\rceil}$ contains $C$.

Now Theorem~\ref{t: normal} gives the result.
\end{proof}

\subsection{$SL_n(2)$}\label{s: sl}

This section is devoted to a proof of Proposition~\ref{p: sl}. We first need a result of Tiep and Zalesskii \cite{tz}.

\begin{lem}\label{l: tz}
 Let $G=SL_n(2)$ with $n\geq 6$. Let $\chi_1$ (resp. $\chi_2$) be the non-trivial complex representation of smallest (resp. second smallest) degree. Then
$$\chi_1(1) = 2^n-2, \, \, \, \chi_2(1) = \frac13(2^n-1)(2^{n-1}-4).$$
\end{lem}

\begin{cor}\label{c: tz}
 Let $H$ be a maximally split torus in $G$. Then $$d_H = \frac13(2^n-1)(2^{n-1}-4).$$
\end{cor}
\begin{proof} 
Observe first that $|H|<\chi_2(1)$. Next we show that $\langle 1_H^G, \chi_1\rangle = 0$.

Consider the action of $G$ on non-trivial vectors in the natural module. The stabilizer of a point in this action is a parabolic subgroup $P=P_1$. Since this action is $2$-transitive we conclude that $1_P^G = 1_G + \pi$ for some irreducible complex representation of degree $2^n-2$. Thus $\pi=\chi_1$ and $1_P^G=1_G+\chi_1$.

Let $K$ be a subgroup of $G$ and consider $1_K^G$. Using Frobenius reciprocity we have
$$\langle 1_K^G, 1_P^G\rangle = \langle 1_K, (1_P^G)|_H\rangle = \langle 1_K, 1_{P\cap K}^K\rangle.$$
Now $\langle 1_K, 1_{P\cap K}^K\rangle$ is equal to the number of orbits of $K$ on the non-trivial vectors in the natural module. Now consider the situation when $K=H$, a maximally split torus. Then $H$ has a single orbit on non-trivial vectors and so 
$$\langle 1_H^G, 1_P^G\rangle = 1.$$
Since $\langle 1_H^G, 1_G \rangle = 1$ we conclude that $\langle 1_H^G, \chi_1\rangle = 0$.

\end{proof}

We are ready to prove Proposition~\ref{p: sl}.

\begin{proof}
Once again observe that the result is trivial for $n$ less than any absolute constant and assume from here on that $n>100$. Set $H=C_G(g)$, a maximally split torus, and observe that Corollary~\ref{c: tz} implies that $d_H=\frac13(2^n-1)(2^{n-1}-4)$.

Since $|H|=2^n-1$, we apply the pigeonhole principle to the cosets of $H$ to conclude that $|(S\cup S^{-1})^3|$ exceeds
$$\left|\bigcup\limits_{s,g\in S} sgs^{-1}\right| \geq \frac{1}{2^n-1}\left(\frac{1}{d_H}\right)^{1-\alpha}|C| > \frac1{2^n-1}\left(\frac{1}{d_H}\right)|C|>\left(\frac{1}{d_H}\right)^3|C|.$$


Thus we can apply Lemma~\ref{l: simples} to the set $(S\cup S^{-1})^3$ to conclude that 
$$|(S\cup S^{-1})^{6\lceil\frac{2}{\alpha}\rceil - 1} \cap C| \geq \frac{|C|}{2}.$$
Let $A_1$ be the set $(S\cup S^{-1})^{6\lceil\frac{2}{\alpha}\rceil-1} \cap C$ and observe that, since $n\geq 100$, we have $m_c>\frac{8}{49}|H|^2$. Then Corollary~\ref{c: trick 3}, applied to $A_1$ with $k=49$, implies that $(S\cup S^{-1})^{3000\lceil\frac{2}{\alpha}\rceil}$ contains $C$.

Now Theorem~\ref{t: normal} gives the result.
\end{proof}

\section{Further work}\label{s: further work}

There is plenty of scope for further work.

\subsection{Quasirandom group actions}

Clearly \cite{gowers} is an El Dorado of a paper and we have mined but a small portion of it for our inspiration here. The latter parts of the paper (which we have neglected) put the notion of a {\it $d$-quasirandom group} on a firm footing, and present a number of different ways of characterising such groups. 

Our work suggests that these ideas belong more properly in the more general setting of {\it $d$-quasirandom group actions}. We have not defined this notion formally in the body of the paper, however, as it is not entirely clear what the `correct' definition should be. Theorem~\ref{t: main convolution} suggests that a transitive group action should be called {\it $d$-quasirandom} if $d_H\geq d$. This would interact well, for instance, with our treatment of expanders, as per Corollary~\ref{c: end game}.

The problem is that Theorem~\ref{t: main convolution 2} also implies mixing properties for a different class of large set. The following lemma suggests that, since Theorem~\ref{t: main convolution 2} is expressed in terms of $m_H$ rather than $d_H$, one might suspect that the bound it specifies is weaker than Theorem~\ref{t: main convolution}.

\begin{lem}\label{l: qr}
Let $J<H<G$ and let $\chi$ be an irreducible character of $G$. Then $\langle 1_J^G, \chi \rangle>\langle 1_H^G, \chi\rangle$. 
\end{lem}
\begin{proof}
We use Frobenius reciprocity:
\begin{equation}\label{e: increase}
\langle 1_J^G, \chi\rangle = \langle (1_J^H)_H^G, \chi\rangle = \langle 1_J^H, \chi|_H\rangle \geq \langle 1_H, \chi|_H\rangle =\langle 1_H^G, \chi\rangle.
\end{equation}
\end{proof}

(Note that, when $J=\{1\}$, $1_J^G$ is the (right) regular permutation character, and so $\langle 1_J^G, \chi\rangle = \dim(\chi)$. In particular we obtain that $m_H\leq d_H$.)

Working in favour of Theorem~\ref{t: main convolution 2}, however, is the fact that $\ell_S\leq |G|/|\Omega|$. Thus, for particular sets $S$ it is conceivable that Theorem~\ref{t: main convolution 2} will be stronger than Theorem~\ref{t: main convolution 2}.

A further complicating factor is that, although Theorem~\ref{t: main convolution} can be rewritten using Lemma~\ref{l: 32}, so that it is stated in terms of arbitrary functions that sum to 0 (see Proposition~\ref{p: expanders}), the reverse process cannot be applied to Theorem~\ref{t: main convolution}. The method of proof for Theorem~\ref{t: main convolution} uses specific properties of $\chi_S$ and does not admit (obvious) generalization to arbitrary measures on the set $S$.

\subsection{The quantity $\ell_C$}

The considerations just discussed suggest that the size of the quantity $\ell_S$ should have a bearing in attempts to understand how quasirandomness interacts with arbitrary group actions. 

Let us focus on the case when $G$ acts by conjugation on a conjugacy class $C$. In this case $\ell_S$ is bounded above by the quantity
$$\ell_C=\max\{ C\cap gH \, \mid \, g\in G\}.$$
(Here $H$ is the centralizer of an element of $C$.)

The computation of $\ell_C$ would seem potentially more tractable than the computation of $\ell_S$ for arbitrary $S$. (Indeed to make use of our results one only needs an upper bound on $\ell_C$.) However we have been unable to make any general statements other than the obvious one: $\ell_C \leq |H|$.

There is reason to believe that better bounds hold. For instance, for the cases discussed in \S\S\ref{s: alternating group} and \ref{s: sl}, we have the following conjectures.

\begin{conj}\label{c: alternating group}
 Let $G=A_n$ with $n$ odd and let $C$ be a conjugacy class of $n$-cycles with $H$ a centralizer of an element of $C$. Then
$$\max\{|gH\cap C| \, \mid \, g\in G\} = |H\cap C| = |N_G(H): H|\in \{\phi(n), \phi(n)/2\}.$$
\end{conj}

Here $\phi$ is Euler's totient function.

\begin{conj}\label{c: sl2}
 Let $G=SL_n(2)$ and let $C$ be a conjugacy class of elements centralized by a maximally split torus, and let $H$ be a centralizer of an element of $C$. Then
$$\max\{|gH\cap C| \, \mid \, g\in G\} = |H\cap C| = |N_G(H): H| =n.$$
 \end{conj}

In both conjectures the first equality is the difficult one. In both cases, too, the equality has been verified using GAP and MAGMA for small values of $n$ \cite{gap, magma}. A proof of these conjectures would immediately yield stronger versions of Propositions~\ref{p: alternating group} and \ref{p: sl}.

\subsection{Minimally quasirandom actions}

The results listed in \S\ref{s: simple} demonstrate that Theorem~\ref{t: main result} can be applied to actions other than the (left) regular action of a group on itself. How many other such actions exist?

In order to answer this question we need to exclude some obvious redundancy. Observe first that Lemma~\ref{l: qr} implies that if $H<N<G$, then $d_N\leq d_H$. Consider what happens when $d_N=d_H$: the bounds given in Theorem~\ref{t: main result} apply equally to the action of $G$ on cosets of $H$, as well as on cosets of $N$. However, in a sense, the growth in the action of $N$ is simply a function of growth on the cosets of $H$, and is of its limited interest in its own right.

We propose, then, the following definition. We write $1< d_1 < d_2 < \dots$ for the degrees of the irreducible characters of $G$ and, for $i$ a positive integer, we say that $(G,H)$ is an {\it $i$-minimal $\qr$-action} if the following conditions are satisfied:
\begin{enumerate}
 \item the minimal degree of a non-trivial component of $1_H^G$ is at least $d_i$;
\item if $F<H$ then the minimal degree of a non-trivial component of $1_H^G$ is strictly less than $d_i$;
\item $d_i>|H|$.
\end{enumerate}

If $G$ is perfect, i.e. $G=[G,G]$; then all non-trivial characters of $G$ have degree strictly greater than $1$ and we conclude that $(G,\{1\})$ is the only $1$-minimal $\qr$-action. This is the action to which the original Gowers trick applied. It is easy to check that the actions $(G,H)$ discussed in \S\S\ref{s: alternating group} and \ref{s: sl} are $2$-minimal $\qr$-actions. Now the question remains: can we classify all such actions for all simple groups, indeed for all perfect groups?

\providecommand{\bysame}{\leavevmode\hbox to3em{\hrulefill}\thinspace}
\providecommand{\MR}{\relax\ifhmode\unskip\space\fi MR }
\providecommand{\MRhref}[2]{%
  \href{http://www.ams.org/mathscinet-getitem?mr=#1}{#2}
}
\providecommand{\href}[2]{#2}

\end{document}